\documentclass[12pt,reqno]{amsproc}
\usepackage{times}

\usepackage{amsmath,amsthm,amssymb,wasysym,mathtools}
\usepackage[inline]{enumitem}
\usepackage{caption}
\usepackage{subcaption}
\usepackage{graphicx}

\DeclareMathOperator{\RE}{Re}

\numberwithin{equation}{section}
\theoremstyle{plain}

\newtheorem{theorem}{Theorem}[section]
\newtheorem{lemma}[theorem]{Lemma}
\newtheorem{corollary}[theorem]{Corollary}

\theoremstyle{definition}

\allowdisplaybreaks

\begin{document}\vspace*{-0.25cm}
\title[The Booth Lemniscate Starlikeness Radius]{The Booth Lemniscate Starlikeness Radius\\ for Janowski Starlike Functions}
		
\author[S. Malik]{Somya Malik}
\address{Department of Mathematics \\National Institute of Technology\\Tiruchirappalli-620015,  India }
\email{arya.somya@gmail.com}

\author[R. M. Ali]{Rosihan M Ali}
\address{School of Mathematical Sciences\\
Universiti Sains Malaysia\\
11800 USM Penang\\ Malaysia}
\email{rosihan@usm.my}

	\author[V. Ravichandran]{V. Ravichandran} 
	\address{Department of Mathematics \\National Institute of Technology\\Tiruchirappalli-620015,  India }
	\email{vravi68@gmail.com; ravic@nitt.edu}
	
	\begin{abstract} The function $G_\alpha(z)=1+ z/(1-\alpha z^2)$, \, $0\leq \alpha <1$, maps the open unit disc $\mathbb{D}$ onto the interior of a domain known as the Booth lemniscate. Associated with this function $G_\alpha$ is the recently introduced class $\mathcal{BS}(\alpha)$ consisting of normalized analytic functions $f$ on $\mathbb{D}$ satisfying the subordination $zf'(z)/f(z) \prec G_\alpha(z)$. Of interest is its connection with known classes $\mathcal{M}$ of functions in the sense $g(z)=(1/r)f(rz)$ belongs to $\mathcal{BS}(\alpha)$ for some $r$ in $(0,1)$ and all $f \in \mathcal{M}$. We find the largest radius $r$ for different classes $\mathcal{M}$, particularly when $\mathcal{M}$ is  the  class of starlike functions  of order $\beta$, or the Janowski class of starlike functions. As a primary tool for this purpose, we find the radius of the largest disc contained in  $G_\alpha(\mathbb{D})$ and centered at a certain point $a \in \mathbb{R}$.
	\end{abstract}
	
	\thanks{ The first author is supported by the UGC-JRF Scholarship. The  second   author gratefully acknowledge support from a USM research university grant 1001.PMATHS.8011101.}
	
	\maketitle
	
	\section{Introduction}
	Let $\mathcal{A}$ be the class of functions analytic on the open unit disc $\mathbb{D}:=\{z\in\mathbb{C}:|z|<1\}$ and   normalized by  $f(0)= f'(0)-1= 0$. Further, let $\mathcal{S}$ be its  subclass consisting of univalent functions. An analytic function $f$ is \emph{subordinate} to an analytic function $g$, written  $f\prec g$, if $f(z)=g(w(z))$ for some analytic self-map $w:\mathbb{D}\to \mathbb{D}$ with $w(0)=0$.  When the superordinate function $g$ is univalent,  then $f\prec g$ if and only if $f(0)=g(0)$ and $f(\mathbb{D}) \subseteq g(\mathbb{D})$. Several important subclasses of  $\mathcal{A}$ are defined by $zf'(z)/f(z)$ and $1+zf''(z)/f'(z)$ respectively being subordinate to a function of positive real part. For an analytic function $\varphi:\mathbb{D}\to \mathbb{C}$, Ma and Minda \cite{MaMinda} gave a unified treatment on growth, distortion, covering and coefficient problems for the two subclasses
\[\mathcal{S^{*}}(\varphi):=\left \{f\in \mathcal{A}: \dfrac{zf'(z)}{f(z)} \prec \varphi(z)\right \}\] and
\[\mathcal{K}(\varphi):=\left \{f\in \mathcal{A}:1+ \dfrac{zf''(z)}{f'(z)} \prec \varphi (z)\right \}.\]
Here $\varphi$ is assumed to be univalent with positive real part, $\varphi(\mathbb{D})$ is starlike with respect to $\varphi(0)=1$, symmetric about the real axis and $\varphi'(0)>0$. If $\varphi$ has positive real part, then functions in $\mathcal{S^{*}}(\varphi)$ and $\mathcal{K}(\varphi)$ are starlike and convex respectively, and thus are univalent. Convolution theorems for some general classes were earlier investigated by Shanmugam \cite{shan} under the stronger assumption of convexity imposed on $\varphi$. Radius problems have also been investigated but only for special cases of $\varphi$.

	For $0\leq \alpha <1$,  let $G_{\alpha}:\mathbb{D}\to \mathbb{C}$ be the function defined by  $G_{\alpha}(z)=1+z/(1-\alpha z^2)$, and $\mathcal{BS}(\alpha):=\mathcal{S}^*(G_\alpha)$. This class was introduced by Kargar et al.\ \cite{booth1}. It is worth noting that $\mathcal{BS}(\alpha)$ contains non-univalent functions because $G_\alpha$ is not of positive real part. Functions belonging to the class $\mathcal{BS}(\alpha)$ are called \emph{Booth lemniscate starlike functions of order} $\alpha$. Other properties of $\mathcal{BS}(\alpha)$ have been studied in \cite{booth2,booth3}, while some closely related classes were also studied in \cite{kanas,kanas2}. Recently, Cho et al. \cite{ChoBooth} obtained some subordination and radius results for $\mathcal{BS}(\alpha)$.

 When $\varphi_{A,B}:\mathbb{D}\to\mathbb{C}$ is $\varphi_{A,B}(z)=(1+Az)/(1+Bz)$, $-1\leq B <A\leq 1$, then the class $\mathcal{S}^*(\varphi_{A,B})=:\mathcal{S}^*[A,B]$ is the well-known class of Janowski starlike functions \cite{Jano}. In particular, if   $0\leq \beta<1$, the class $\mathcal{S}^*(\beta):=\mathcal{S}^*[1-2\beta,-1]$ is the class of starlike functions of order $\beta$.  The classes $\mathcal{S}^*=\mathcal{S}^*(0)$ and $\mathcal{K}=\{f\in\mathcal{A}: zf'(z)\in\mathcal{S}^*\}$ are the classical classes of starlike and convex functions.


Let $\mathcal{M}$ be a given class of analytic functions in $\mathcal{A}$. To each  $f \in \mathcal{M}$, let
$$R_{f} = \sup \left \{r:\dfrac{zf'(z)}{f(z)} \in G_{\alpha}(\mathbb{D}),\quad |z|\leq r<1   \right\},$$
and
$$R_{\mathcal{BS}(\alpha)}(\mathcal{M}) = \inf \{R_{f}: f \in \mathcal{M}   \}.$$
The number $R_{\mathcal{BS}(\alpha)}(\mathcal{M})$ is known as the $\mathcal{BS}(\alpha)$-radius or the \emph{Booth lemniscate starlikeness radius of order} $\alpha$ for the class $\mathcal{M}$. We shall use these two terms interchangeably. Thus the function $g(z)=(1/r)f(rz)$ belongs to $\mathcal{BS}(\alpha)$ for every $r \leq R_{\mathcal{BS}(\alpha)}(\mathcal{M}).$

 In this paper, we seek to determine the $\mathcal{BS}(\alpha)$-radius $R_{\mathcal{BS}(\alpha)}(\mathcal{M})$ when $\mathcal{M}$ is the class of starlike functions of order $\beta$, or $\mathcal{M}$ is the class of Janowski starlike functions. As a primary tool, we first obtained in Section 2, the largest disc contained in $G_\alpha(\mathbb{D})$ and centered at a given point $a$, as well as the smallest disc containing $G_\alpha(\mathbb{D})$ and centered at $a$ .

 In Section 3, this result is applied to determine the Booth lemniscate starlikeness radius of order $\alpha$ for the class of starlike functions of order $\beta$. In this section too, the $\mathcal{BS}(\alpha)$-radius is also determined for the class of convex functions and the class consisting of functions $f\in \mathcal{A}$ with $zf'(z)/f(z)$ lying in the half-plane $\{w: \operatorname{Re} w<\beta\}$, \, $1<\beta<4/3$.

 In Section 4, conditions on $A$ and $B$ are determined that will ensure the Janowski functions $f \in \mathcal{S}^*[A,B]$ also belong to the class $\mathcal{BS}(\alpha)$. When these conditions are not met, we find  the Booth lemniscate starlikeness radius for $\mathcal{S}^*[A,B]$. Booth lemniscate starlikeness radius is also deduced for other related classes.



\section{Preliminaries}
Let $\mathbb{D}(a;r):=\{z\in\mathbb{C}:|z-a|<r\}$ be the open disc of radius $r$ centered at $z=a$. If $f\in \mathcal{M}$, then $zf'(z)/f(z)\in 	 \mathbb{D}(a_f(r);c_f(r))$ for $r$ sufficiently small. Thus the $\mathcal{BS}(\alpha)$-radius for the class $\mathcal{M}$ is found by determining the largest disc so that $\mathbb{D}(a_f(r);c_f(r))\subseteq G_{\alpha}(\mathbb{D})$.

For this purpose, a key objective in this section is to find the radius $r_a$ of the largest disc $\mathbb{D}(a;r_a)$ contained in $G_\alpha(\mathbb{D})$ and centered at a given point $a$. We also find the radius $R_a$ of the smallest disc $\mathbb{D}(a;R_a)$ containing $G_\alpha(\mathbb{D})$ centered at $a$.  Since the range of $zf'(z)/f(z)$ contains the point $1$ for any $f\in \mathcal{A}$, we may assume that the center $a$ of the disc satisfy the inequality
   \[ \frac{1-2\alpha}{2-2\alpha}< a < \frac{3-2\alpha}{2-2\alpha}.\]
This will ensure that the disc $\mathbb{D}(a;r_a)$ contains the point $w=1$.
		

In \cite[Lemma 3.4]{ChoBooth}, Cho et al.\ found the largest disc centered at $w=1$ contained in $G_\alpha(\mathbb{D})$ and the smallest disc centered at $w=1$ containing $G_\alpha(\mathbb{D})$. Specifically, they showed that 	
	\begin{equation*}
		\mathbb{D}(1;1/(1+\alpha)) \subseteq G_{\alpha}(\mathbb{D}) \subseteq   \mathbb{D}(1; 1/(1-\alpha)).
	\end{equation*}
This readily follows since
\[ \frac{1}{1+\alpha} \leq |G_\alpha(e^{it})-1|=\frac{1}{|1-\alpha e^{2it}|} \leq \frac{1}{1-\alpha}. \]
Here, we compute the radii of these two discs when the centers are located at an arbitrary point $a \in  \mathbb{R}\cap G_{\alpha}(\mathbb{D})$.
	
	\begin{lemma}\label{ralemma} Let $0\leq \alpha<1$ and $(1-2\alpha)/(2-2\alpha)<a<(3-2\alpha)/(2-2\alpha)$. Then the following inclusions hold:

		\begin{equation*}
			\mathbb{D}(a;r_{a})\subseteq G_{\alpha}(\mathbb{D})
			\subseteq \mathbb{D}(a;R_{a}),
		\end{equation*}
where  $r_a$ and $R_a$  are  given  by
\begin{align*}
			r_{a} &=
			\begin{dcases}
				a-1+\dfrac{1}{1-\alpha},  &   \frac{1-2\alpha}{2-2\alpha}< a
				\leq 1-\frac{4\alpha }{(1-\alpha)(1+6\alpha+\alpha^2)},\\
				\sqrt{s(\alpha, a)}, &
            1-\frac{4\alpha }{(1-\alpha)(1+6\alpha+\alpha^2)}
				<  a< 1+\frac{4\alpha }{(1-\alpha)(1+6\alpha+\alpha^2)},\\
				1-a+\dfrac{1}{1-\alpha},  &
				1+\frac{4\alpha }{(1-\alpha)(1+6\alpha+\alpha^2)}\leq  a <\frac{3-2\alpha}{2-2\alpha},
			\end{dcases} \intertext{and}
			R_{a}&=\begin{dcases}
				1-a+\dfrac{1}{1-\alpha},
				&  \frac{1-2\alpha}{2-2\alpha}<a\leq 1,\\
				a-1+\dfrac{1}{1-\alpha}, &   1\leq a< \frac{3-2\alpha}{2-2\alpha},
			\end{dcases}
			\intertext{with}
			s(\alpha, a) &=\dfrac{\sqrt{\alpha [\alpha- (1-a)^2 (1-\alpha ^2)^2]}
				+ \alpha (1+2 (1+\alpha)^2(1-a)^2)}{2\alpha (1+\alpha)^2}, \quad \alpha\neq 0.
		\end{align*}
	\end{lemma}
	
\begin{proof} As noted earlier, the result for  $a=1$ was proved in \cite[Lemma 3.4]{ChoBooth}.   If $\alpha=0$, then  $G_{0}(z)=1+z$. In this case, readily $r_a = a$ and $R_a = 2-a$ for $1/2<a\leq 1$. Also, for $1\leq a< 3/2$, it is readily seen that $	r_a = 2-a$ and $ R_a = a $.

Thus, assume next that $\alpha \neq 0$ and   $a\neq 1$. The boundary $\partial G_{\alpha}(\mathbb{D})$ of the image of the unit disc $\mathbb{D}$ in parametric form is given by
\[G_{\alpha}(e^{it})=1+\frac{e^{i t}}{1-\alpha  e^{2 i t}}= 1+\frac{(1-\alpha ) \cos t+i (1+\alpha ) \sin t}{1+\alpha ^2-2 \alpha  \cos (2 t) } .\]
The result is proved by showing the minimum and maximum distance from the point $(a,0)$ to the point on the boundary $\partial G_{\alpha}(\mathbb{D})$  are respectively $r_a$ and $R_a$.

Thus consider the function
		\begin{align}\label{eqn23a} H(\cos t)&=\notag
			\left(1-a+\frac{(1-\alpha ) \cos t }{1+\alpha ^2-2 \alpha  \cos (2 t) }\right)^2+ \left( \frac{  (1+\alpha ) \sin t}{1+\alpha ^2-2 \alpha  \cos (2 t) }\right)^2\\
			&= (1-a)^2+\dfrac{1+2(1-a)(1-\alpha) \cos(t)}{1+\alpha ^2-2\alpha \cos(2t)},
			\intertext{that is,}
			\label{eqn23}
			H(x) &= (1-a)^2+\dfrac{1+2(1-a)(1-\alpha)x}{(1+\alpha)^2-4\alpha x^2},\quad  x=\cos t\in [-1,1].
			\intertext{A computation using \eqref{eqn23} shows that } \notag
			H'(x) & =  \frac{2 (1-a) (1-\alpha )}{(\alpha +1)^2-4 \alpha  x^2}+\frac{8 \alpha  x (2 (1-a) (1-\alpha ) x+1)}{\left((\alpha +1)^2-4 \alpha  x^2\right)^2}\\  \label{eqn23b}
			& = \frac{8\alpha  (1-\alpha) (1-a)(x-x_1)(x-x_2) }{\left((\alpha +1)^2-4 \alpha  x^2\right)^2},
			\intertext{where $x_1$ and $x_2$ are the two zeros of $H'(x)$. These are the zeros of the polynomial } \label{eqn2p}
			& 4\alpha(1-\alpha)(1-a)x^2+4\alpha x +(1-\alpha)(1+\alpha)^2(1-a)=0
			\intertext{given by}\notag
			x_{1}& =-\dfrac{\alpha+\sqrt{\alpha(\alpha - (1-a)^2 (1-\alpha^2)^2)}}{2\alpha(1-a)(1-\alpha)}, \\
			\intertext{and}\notag
			x_{2}& =-\dfrac{\alpha-\sqrt{\alpha(\alpha - (1-a)^2 (1-\alpha^2)^2)}}{2\alpha(1-a)(1-\alpha)}.
		\end{align}
		
The zeros $x_1$ and $x_2$ satisfy  $x_1x_2=(1+\alpha)^2/(4\alpha)\geq 1$. Further, $x_1$, $x_2$  are real if   $\alpha$ and $a$ satisfy
		$\alpha \geq (1-a)^2 (1-\alpha^2)^2$, or equivalently, whenever
		\begin{equation*}
			|a-1|\leq \dfrac{\sqrt{\alpha}}{1-\alpha ^2}.
		\end{equation*}
		
		
The following notations are introduced to give greater clarity to the proof. Let
		\begin{align*}
			\alpha_0 & := 1-\frac{\sqrt{\alpha}}{1-\alpha^2}, \quad
			\alpha_1   := 1-\frac{4\alpha }{(1-\alpha)(1+6\alpha+\alpha^2)},\\
			\tilde{\alpha}_0&:= 1+\frac{\sqrt{\alpha}}{1-\alpha^2},
			\quad
			\tilde{\alpha}_1 :=1+\frac{4\alpha }{(1-\alpha)(1+6\alpha+\alpha^2)}.
		\end{align*}
A little computations shows that $ x_1  <-1$ for $\alpha_0 <a<1$, while $x_1  >1 $ for  $1<a<\tilde{\alpha}_0$. Similarly, $x_2<0$ for $a<1$; indeed, $x_2<-1$ for $\alpha_0 \leq a <\alpha_1$, and  $-1\leq x_2\leq 0$ for $\alpha_1 \leq a<1 $. Also, $x_2>0$ for $a>1$; indeed,  $0\leq x_2\leq 1$ for $ 1<a\leq\tilde{\alpha}_1$ and $x_2>1$ for $\tilde{\alpha}_1 < a \leq \tilde{\alpha}_0$. These observations together with \eqref{eqn23b} will be helpful in the following cases.
		
		Case (i). If $\alpha_0 \leq a \leq \alpha_1$, it follows that the function $H$ is increasing and therefore
		\[r_a =\sqrt{H(-1)}=a-1+\dfrac{1}{1-\alpha}, \quad \text{and}\quad R_a =\sqrt{H(1)}=1-a+\dfrac{1}{1-\alpha}.\]
		
		Case (ii). If $\alpha_1  < a<1$, then $H'(x)<0$ for $x<x_2$, while $H'(x)>0$ for $x>x_2$. Thus, $x_2$ is the minimum point.  Since $a<1$, the maximum of $H$ occurs at $x=1$. Therefore,
\[r_a =\sqrt{H(x_2)}\quad \text{and}\quad R_a =\sqrt{H(1)}.\]
Note that
		\begin{align*}
			\sqrt{H(x_2)}  &=\sqrt{\dfrac{\sqrt{\alpha [\alpha- (1-a)^2 (1-\alpha ^2)^2]}+\alpha (1+2 (1+\alpha)^2(1-a)^2)}{2\alpha (1+\alpha)^2}}.
		\end{align*}
		
		Case (iii). If $ 1<a< \tilde{\alpha}_1$, then  $H'(x)<0$ for $x<x_2$, while $H'(x)>0$ for $x>x_2$. Thus, $x_2$ is the minimum point. Since $a>1$, the function $H$ attains its maximum at $x=-1$ so that
		\[r_a =\sqrt{H(x_2)}\quad \text{and}\quad R_a =\sqrt{H(-1)}.\]
		
		Case (iv). If  $\tilde{\alpha}_1 \leq  a \leq \tilde{\alpha}_0$,  then the function  $H$ is  decreasing, whence
		\[r_a =\sqrt{H(1)}\quad \text{and}\quad R_a =\sqrt{H(-1)}. \]

It remains next to consider the range
\[|a-1|> \sqrt{\alpha}/(1-\alpha ^2).\]
In this case, $H'$ is non-vanishing in $[-1,1]$. Since
		\begin{equation*}
			H'(0)=\frac{8\alpha  (1-\alpha) (1-a)x_1 x_2}{(\alpha +1)^4},
		\end{equation*}
and \eqref{eqn2p} yields $x_1 x_2= (1+\alpha)^2/(4\alpha) >0$, it follows that  $H'(0)<0$ for $a>1$, while $H'(0)>0$ for $a<1$. Since $H'$ is non-vanishing, we deduce for $x\in[-1,1]$ that $H'(x)<0$ whenever $a>1$, and $H'(x)>0$ for $a<1$. Therefore, for $a>1$,
		\[ r_a=\min\sqrt{H(x)}= \sqrt{H(1)} \quad \text{and}\quad R_a=\max\sqrt{H(x)} = \sqrt{H(-1)}.\]

Similarly,  for  $a<1$,
\[r_a=\ \sqrt{H(-1)}\quad \text{and}\quad R_a=\sqrt{H(1)}.\]
The results now follow because
		\begin{align*}
			\sqrt{H(-1)}& = a-1+\dfrac{1}{1-\alpha}, \quad \text{and}\quad
			\sqrt{H(1)}  =\ 1-a+\dfrac{1}{1-\alpha}. \qedhere
		\end{align*}		
\end{proof}

It is worth noting that the assumption in Lemma~\ref{ralemma} on the center $a$ ensures that the disc $\mathbb{D}(a;r_{a})$ contains the point $w=1$.

\section{Starlike functions of order $\beta$}

 A function $f\in \mathcal{A}$ is \emph{starlike} if $tf(z) \in f(\mathbb{D})$ whenever $0 \leq t \leq 1.$ Analytically, this is equivalent to the condition $\RE(zf'(z)/f(z))>0$ for all $z\in \mathbb{D}$. A generalization is a function $f\in \mathcal{A}$ satisfying $\RE(zf'(z)/f(z))>\beta$ for $z\in \mathbb{D}$, where $0\leq \beta <1$. This function is known as \emph{starlike of order }$\beta$, and the class consisting of such functions is denoted by $\mathcal{S}^*(\beta)$. In terms of subordination, $f\in \mathcal{S}^*(\beta)$ is subordinate to the function $(1+(1-2\beta)z)/(1-z)$. It is readily seen that the function
\begin{equation}\label{extremal}
k_{\beta}(z)=\dfrac{z}{(1-z)^{2-2\beta}}.
\end{equation}
satisfies the equation  $zf'(z)/f(z)=(1+(1-2\beta)z)/(1-z)$. This function $k_\beta$ is called the generalized Koebe function. It serves as the extremal function for the radius problem considered in the next theorem.


\begin{theorem}\label{th1}
Let $0<\alpha<1$ and $0\leq \beta <1$. If $f\in \mathcal{S}^*(\beta)$, then $f$ is Booth lemniscate starlike of order $\alpha$ in the disk of radius

\begin{align}\label{eqn31}
			R_{\mathcal{BS}(\alpha)}(\mathcal{S^{*}(\beta)})
			&=\begin{dcases}				
\dfrac{2\sqrt{\alpha}}{(1+\alpha)\sqrt{ 1+16\alpha (1-\beta)^2}},
				& \ 0\leq \beta<\max\left\{0;  \dfrac{9\alpha-1}{8\alpha}\right\}, \\
				\dfrac{1}{1+2(1-\alpha)(1- \beta)}, & \  \max\left\{0;   \dfrac{9\alpha-1}{8\alpha}\right\}\leq   \beta< 1.
			\end{dcases}
		\end{align}
	\end{theorem}

	\begin{proof} It is known (see \cite{ravi}) that functions $f\in \mathcal{S^{*}(\beta)}$ satisfy
		\begin{equation*}
			\left|\dfrac{zf'(z)}{f(z)}-\dfrac{1+(1-2\beta)r^2}{1-r^2}\right| \leq \dfrac{2(1-\beta)r}{1-r^2}, \quad |z|\leq r<1.
		\end{equation*}
		Thus $zf'(z)/f(z)\in 	\mathbb{D}(a_f(r);c_f(r))$ where
\begin{equation}\label{eqn32x} a_f(r):=\dfrac{1+(1-2\beta)r^2}{1-r^2} \quad \text{ and }\quad c_f(r):= \dfrac{2(1-\beta)r}{1-r^2}.
\end{equation}
		
We wish to find $\rho$ so that $\mathbb{D}(a_f(\rho);c_f(\rho))\subseteq G_{\alpha}(\mathbb{D})$ for every $f\in \mathcal{S^{*}(\beta)}$. Since $a_f'(r)>0$, we note that $a_f(r)\geq 1$. Now Lemma \ref{ralemma} shows that
		the disc $\mathbb{D}(a_f(r);c_f(r))\subset G_\alpha(\mathbb{D})$ provided

		\begin{align}\label{eqn33a}
			c_f(r)&=\begin{dcases}
				\sqrt{s(\alpha,a)}, &   a_f(r)<1+\frac{4\alpha }{(1-\alpha)(1+6\alpha+\alpha^2)},\\
				1-a_f(r)+\dfrac{1}{1-\alpha},  & \
				1+\frac{4\alpha }{(1-\alpha)(1+6\alpha+\alpha^2)}\leq a_f(r),
			\end{dcases}
			\intertext{where }\label{saf}
			s(\alpha,a_f(r))&=\dfrac{\sqrt{\alpha [\alpha- (1-a_f(r))^2 (1-\alpha ^2)^2]}
				+ \alpha (1+2 (1+\alpha)^2(1-a_f(r))^2)}{2\alpha (1+\alpha)^2}.
		\end{align}

Let us write
\begin{alignat*}{3}
\rho_0 : =	\dfrac{2\sqrt{\alpha}}{(1+\alpha)\sqrt{ (1+16\alpha (1-\beta)^2)}}
 &  \quad\text{ and }\quad &
\tilde{\rho}_0 & :=\dfrac{1}{1+2(1-\alpha)(1- \beta)}.
\end{alignat*}
The  number $\rho_0<1$ is  the positive  root of the equation
		$c_f(r)=\sqrt{s(\alpha,a_f(r))}$ given by \eqref{saf}. Indeed, this equation has the form
\[(2\alpha (1+\alpha)^2(c_f(r)^2-(1-a_f(r))^2))^2= \alpha^2- \alpha (1-\alpha ^2)^2(1-a_f(r))^2, \]
which upon solving and replacing $a_f$ and $c_f$ by the expressions given by \eqref{eqn32x}, yields the solution $\rho_0$.

Also, the number		$\tilde{\rho}_0<1$ is the root of the equation
		\[c_f(r)=1-a_f(r)+\dfrac{1}{1-\alpha}.  \]
Further,
		\[\rho_1:= \dfrac{\sqrt{2\alpha}}{\sqrt{2\alpha+(1-\alpha)(1-\beta)(1+6\alpha+\alpha^2)}}<1\]
is  the positive root of the equation
		\[\dfrac{(1-\beta)r^2}{1-r^2}=\frac{2\alpha }{(1-\alpha)(1+6\alpha+\alpha^2)},\]
obtained from rewriting  the equation \[a_f(r)= 1+\frac{4\alpha }{(1-\alpha)(1+6\alpha+\alpha^2)}.\]

Let us also write
		\begin{alignat*}{3}
			\beta_0 & := 1- \dfrac{1-\alpha}{8\alpha}=\dfrac{9\alpha -1}{8\alpha}, & \quad  \text{   }\quad
            & \tilde{\beta}_0  := 1-\dfrac{1-\alpha}{2(1+\alpha)^2}=\dfrac{1+5\alpha+2\alpha ^2}{2(1+\alpha)^2}.
		\end{alignat*}
Then, $\rho_0=\tilde{\rho_0}$ if $\beta=\beta_0$. Since $0<\alpha<1$, a calculation shows  that $\beta_0<\tilde{\beta}_0$.		

For $\alpha<1/9$, or equivalently for $\beta_0<0$, we shall show that  $R_{\mathcal{BS}(\alpha)}=\tilde{\rho}_0$ for all $0\leq \beta<1$. When $\alpha\geq 1/9$, or equivalently $\beta_0\geq 0$, we shall show that  $R_{\mathcal{BS}(\alpha)}=\rho_0$ for $ 0\leq \beta<\beta_0$, while $R_{\mathcal{BS}(\alpha)}=\tilde{\rho}_0$ for $\beta_0\leq \beta<1$. Thus there are two cases to consider: $ 0\leq \beta< \beta_0$, and $  \beta_0\leq \beta<1$.

%
		
Case (i). Let $ 0\leq \beta< \beta_0$.  A little calculations shows that the  inequality  $\rho_0<  \rho_1$ is equivalent to
		\begin{equation}\label{eqn35}
			(2(1+\alpha)^2 \beta-(1+5\alpha +2\alpha ^2))(1-9\alpha+8\alpha\beta) >0.
		\end{equation}
Inequality \eqref{eqn35} holds  if and only if
		\begin{align*}
			\beta <\min\left\{\beta_0,\tilde{\beta}_0\right\}= \beta_0, \quad
			\text{or}\quad
			\beta &> \max\left\{\beta_0,\tilde{\beta}_0\right\}=\tilde{\beta}_0.
		\end{align*}
In this case, $\rho_0\leq \rho_1$ and
		\[ a_f(\rho_0)\leq a_f(\rho_1) =1+\frac{4\alpha }{(1-\alpha)(1+6\alpha+\alpha^2)}.\]
From Lemma~\ref{ralemma}, it follows that $\mathbb{D}(a_f(\rho_0); c_f(\rho_0))\subset G_\alpha(\mathbb{D})$, whence the $\mathcal{BS}(\alpha)$-radius for $\mathcal{S}^*(\beta)$ is at least $\rho_0$.
		
To validate sharpness of  $ {\rho}_0$, we  consider the generalized Koebe function $k_\beta$ given by \eqref{extremal}. We shall   show the existence of   a point on $|z|={\rho}_0$ that is mapped to a point on $\partial G_\alpha(\mathbb{D})$. In other words, we prove that there is some $t$ such that the image of the point $z={\rho}_0\mathit{e} ^{i  t}$ under the map $zk_\beta'(z)/k_\beta(z)$ belongs to  $G_{\alpha}(\mathit{e}^{i  t})$.

For this purpose, let us write $G_{\alpha}(\mathit{e}^{i  t})=u(t)+i  v(t)$. Then $u$ and $v$ satisfy the equation
		\begin{equation}\label{eqn36}
			((u-1)^2+v^2)^2= \left(\dfrac{u-1}{1-\alpha}\right)^2+ \left(\dfrac{v}{1+\alpha}\right)^2.
		\end{equation}
The representation of the function $zk_\beta'(z)/k_\beta(z)$  at $z= \rho_0\mathit{e} ^{i  t}$ in Cartesian coordinates is
		\begin{align*}
			\dfrac{zk_\beta'(z)}{k_\beta(z)}&=\dfrac{1+(1-2\beta)z}{1-z}
			=\frac{1-\overline{z}+(1-2\beta)z-(1-2\beta)|z|^2}{|1-z|^2}\notag \\
			&=\dfrac{1-2\beta {\rho}_0 \cos t- (1-2\beta)({\rho}_0) ^2}{1+({\rho}_0)^2-2{\rho}_0 \cos t} +i  \dfrac{2(1-\beta){\rho}_0 \sin t}{1+({\rho}_0)^2-2{\rho}_0 \cos t}.
		\end{align*}
By taking
		\[u(t)= \dfrac{1-2\beta {\rho}_0 \cos t- (1-2\beta)({\rho}_0) ^2}{1+({\rho}_0)^2-2{\rho}_0 \cos t} \quad \text{and}\quad   v(t)=\dfrac{2(1-\beta){\rho}_0 \sin t}{1+({\rho}_0)^2-2{\rho}_0 \cos t},\]
it is readily seen that
		\begin{align*}
			((u-1)^2+v^2)^2 &=\dfrac{16({\rho}_0)^4 (1-\beta)^4}{(1+({\rho}_0)^2-2{\rho}_0 \cos t)^2},
\intertext{and}
			\left(\dfrac{u-1}{1-\alpha}\right)^2+ \left(\dfrac{v}{1+\alpha}\right)^2 &=\dfrac{4({\rho}_0)^2(1-\beta)^2 \left[ ({\rho}_0-\cos t)^2(1+\alpha)^2 + (\sin t)^2 (1+\alpha)^2 \right]}{(1+({\rho}_0)^2-2{\rho}_0 \cos t)^2 (1-\alpha)^2 (1+\alpha)^2}.
		\end{align*}

From \eqref{eqn36}, we seek to find a $t$ satisfying the equation
		\[4({\rho}_0)^2 (1-\beta)^2 =  \dfrac{({\rho}_0 - \cos t)^2}{(1-\alpha)^2} + \dfrac{(\sin t)^2}{(1+\alpha)^2} .\]
Replacing the value ${\rho}_0$ and writing $x=\cos t$ yields
\begin{align*} & \dfrac{16\alpha (1-\beta)^2}{(1+\alpha)^2 (1+16\alpha (1-\beta)^2)}-\dfrac{1-x^2}{(1+\alpha)^2}\\
& \quad {}-\dfrac{1}{(1-\alpha)^2}\left(x-\dfrac{2\sqrt{\alpha}}{(1+\alpha)\sqrt{1+16\alpha (1-\beta)^2}}\right)^2=0,
\end{align*}
or equivalently, the equation
\begin{equation}\label{eqn38}
4\alpha\left(1 +16\alpha   (1-\beta)^2\right)x^2 -4  (1+\alpha)\sqrt{\alpha (1+16\alpha (1-\beta)^2)}x+(1+\alpha)^2=0.
		\end{equation}

Clearly, the  number
\begin{align*}
x_0 &=\dfrac{1+\alpha}{2\sqrt{\alpha (1+16\alpha (1-\beta)^2)}}
\end{align*}
is the positive double real root of \eqref{eqn38}.
Since 	$ 0<\beta\leq \beta_0$, a  computation shows that $x_0<1$.
With $t=\arccos x_0$ and $z=\rho_0\mathit{e} ^{i  t}$, the point  $zk_\beta'(z)/k_\beta(z)$ lies on $\partial G_{\alpha}(\mathbb{D})$. This proves sharpness for ${\rho}_0$.

Case (ii). Let $  \beta_0\leq \beta<1$. Here $\tilde{\rho}_0\geq \rho_1$ and
		\[ a_f(\tilde{\rho}_0)\geq a_f(\rho_1) =1+\frac{4\alpha }{(1-\alpha)(1+6\alpha+\alpha^2)}.\]
From Lemma~\ref{ralemma}, it follows that $\mathbb{D}(a_f(\tilde{\rho}_0); c_f(\tilde{\rho}_0))\subset G_\alpha(\mathbb{D})$ showing that the $\mathcal{BS}(\alpha)$-radius for the class of starlike functions of order $\beta$ is at least $\tilde{\rho}_0$.
		
		To show sharpness of $\tilde{\rho}_0$,  consider again the generalized Koebe function $k_\beta$ given by \eqref{extremal}.  We shall find a point on $|z|=\tilde{\rho}_0$ such that it is mapped to a point on $\partial G_\alpha(\mathbb{D})$. Evidently,
		\[\dfrac{zk_\beta'(z)}{k_\beta(z)}=1+2(1-\beta)\frac{z}{1-z}.\]
		Since
		\[  \frac{\tilde{\rho}_0}{1-\tilde{\rho}_0} =\frac{1}{2(1-\alpha)(1-\beta)},   \]
evaluating at $z=\tilde{\rho}_0$ gives
		\[\dfrac{zk_\beta'(z)}{k_\beta(z)}= 1+\dfrac{1}{1-\alpha}=\ G_{\alpha} (1)\in \partial G_\alpha(\mathbb{D}). \] This proves sharpness of  $\tilde{\rho}_0$.
%
%
\end{proof}
	
The condition \eqref{eqn31} suggests that the $\mathcal{BS}(\alpha)$-radius in the case $\alpha=0$ is $1/(3-2\beta)$. That this is indeed the case follows easily from Lemma~\ref{ralemma}.

\begin{theorem}\label{th2}
Let $0\leq \beta <1$. The  Booth lemniscate starlikeness radius (of order 0)  for the class of starlike functions of order $\beta$ is  $1/(3-2\beta)$.
	\end{theorem}
	
Theorem~\ref{th1} and Theorem~\ref{th2} also readily yield the following results for  starlike and convex  functions.

\begin{corollary}
Let $0\leq \alpha<1$. The Booth lemniscate starlikeness radius of order $\alpha$  for the class $\mathcal{S}^{*}$ of starlike functions  is

		\begin{align}\label{eqn31b}
			R_{\mathcal{BS}(\alpha)}(\mathcal{S}^{*})
			&=\begin{dcases}	
				\dfrac{1}{3-2\alpha }, &  0\leq \alpha\leq \frac{1}{9},\\			
\dfrac{2\sqrt{\alpha}}{(1+\alpha)\sqrt{ 1+16\alpha }},
				&  \frac{1}{9}\leq\alpha \leq 1.
			\end{dcases}
		\end{align}
\end{corollary}

\begin{corollary}
Let $0\leq \alpha<1$. The Booth lemniscate starlikeness radius of order $\alpha$  for the class $\mathcal{K}$ of convex functions  is

		\begin{align*}\label{eqn31c}
			R_{\mathcal{BS}(\alpha)}(\mathcal{K})
			&=\begin{dcases}	
				\dfrac{1}{2-\alpha }, & 0\leq \alpha\leq \frac{1}{5},\\			
\dfrac{2\sqrt{\alpha}}{(1+\alpha)\sqrt{  1+4\alpha   }},
				&   \frac{1}{5}\leq\alpha \leq 1.
			\end{dcases}
		\end{align*}
\end{corollary}

\begin{proof} Every convex function is also starlike of order $1/2$. Thus the $\mathcal{BS}(\alpha)$-radius is at least as big as that given by Lemma~\ref{ralemma} with $\beta=1/2$. However, the extremal starlike function $k_{1/2}$  given by \eqref{extremal} is also convex, whence the result.
\end{proof}

Next let $1<\beta<4/3$, and $M(\beta)$ be the class consisting of functions $f\in \mathcal{A}$ for
which $\RE (zf'(z)/f(z))<\beta$. This class was introduced by  Uralegaddi et al. \cite{ural} who investigated functions in the class with positive coefficients. The following result gives the $\mathcal{BS}(\alpha)$-radius for the class $M(\beta)$.

\begin{theorem}\label{th3}
Let $0< \alpha< 1$ and $1< \beta <4/3$. The Booth lemniscate starlikeness radius of order $\alpha$ for the class $M(\beta)$ is
\begin{align*}
			R_{\mathcal{BS}(\alpha)}(M(\beta)) &=\begin{dcases}	
				\dfrac{1}{1+2(1-\alpha)( \beta-1)}, & 1< \beta\leq 1+ \dfrac{1-\alpha}{8\alpha},	\\		
\dfrac{2\sqrt{\alpha}}{(1+\alpha)\sqrt{ 1+16\alpha ( \beta-1)^2}},
				&    1+ \dfrac{1-\alpha}{8\alpha}\leq  \beta< \frac{4}{3}.
			\end{dcases}
		\end{align*}
\end{theorem}
\begin{proof} Every function $f\in M(\beta)$ satisfies the inequality
		\begin{equation*}
			\left|\dfrac{zf'(z)}{f(z)}-\dfrac{1+(1-2\beta)r^2}{1-r^2}\right| \leq \dfrac{2( \beta-1)r}{1-r^2}, \quad |z|\leq r<1.
		\end{equation*}
Define $a_f$   and $c_f$   by
\[ a_f(r):=\dfrac{1+(1-2\beta)r^2}{1-r^2}\quad \text{ and }\quad c_f(r):=\dfrac{2( \beta-1)r}{1-r^2}.\]
As $\beta>1$, it follows that  $a_f$  is decreasing, whence $a_f(r)\leq 1$ for all $0\leq r<1$. Recall that this function was increasing in the case of starlike functions of order $\beta$. Since $a_f(r)\leq 1$, Lemma \ref{ralemma} shows that
the disc $\mathbb{D}(a_f(r);c_f(r))\subset G_\alpha(\mathbb{D})$ provided
\begin{align}\label{eq23}
	c_f(r)&=\begin{dcases}
		\sqrt{s(\alpha,a)}, &   a_f(r)>1-\frac{4\alpha }{(1-\alpha)(1+6\alpha+\alpha^2)},\\
		a_f(r)-1+\dfrac{1}{1-\alpha},  &
		1-\frac{4\alpha }{(1-\alpha)(1+6\alpha+\alpha^2)}\geq a_f(r),
	\end{dcases}
\end{align}
where $s(\alpha,a_f(r))$ is given by \eqref{saf}.

Let \begin{alignat*}{3}
	\rho_0:= \dfrac{1}{1+2(1-\alpha)( \beta-1)}
	&  \quad\text{ and }\quad &
	\tilde{\rho}_0:=
	\dfrac{2\sqrt{\alpha}}{(1+\alpha)\sqrt{ (1+16\alpha ( \beta-1)^2)}}.
\end{alignat*}
Then, $\rho_0$  satisfies the equation
 \begin{align*}
 	 c_f(r)& =a_f(r)-1+\dfrac{1}{1-\alpha},
 \end{align*}
while $\tilde{\rho_0}$ is the solution of the equation
\[c_f(r)^2  = s(\alpha,a_f(r)).\]

Also, \[\rho_1=\dfrac{\sqrt{4\alpha}}{\sqrt{4\alpha+(2\beta-2)(1-\alpha)(1+6\alpha+\alpha ^2)}}\]
is the positive root of the equation\[a_f(r)=1-\dfrac{4\alpha}{(1-\alpha)(1+6\alpha+\alpha ^2)}.\]

Evidently, $\rho_1\leq \rho_0$ holds if and only if
\[\beta\leq 1+\dfrac{1-\alpha}{8\alpha}.\]

Case (i): $1<\beta\leq 1+((1-\alpha)/8\alpha)$. Here $\rho_1\leq \rho_0$, and because the center $a_f(r)$ is decreasing, then
\[a_f(\rho_0)\leq a_f(\rho_1)=1-\dfrac{4\alpha}{(1-\alpha)(1+6\alpha+\alpha ^2)}.\]
Thus, it follows from \eqref{eq23} that $\mathbb{D}(a_f(\rho_0);c_f(\rho_0))\subset G_\alpha(\mathbb{D})$ for every $f\in M(\beta)$, or the $\mathcal{BS}(\alpha)$-radius for $M(\beta)$ is at least $\rho_0$.

Case (ii): $1+((1-\alpha)/8\alpha)\leq \beta <4/3$. In this case, $\rho_1\geq \rho_0$, and because the center $a_f(r)$ is decreasing, then \[a_f(\rho_0)\geq a_f(\rho_1)=1-\dfrac{4\alpha}{(1-\alpha)(1+6\alpha+\alpha ^2)}.\]
Thus, $\mathbb{D}(a_f(\tilde{\rho_0});c_f(\tilde{\rho_0}))\subset G_\alpha(\mathbb{D})$ from \eqref{eq23}.

To complete the proof, we observe that the function  $k_\beta$ given by $k_\beta(z)=z/(1-z)^{2-2\beta}$ shows that the radius in each case above is best possible.
\end{proof}

\section{Janowski starlike functions}

Let $-1\leq B <A\leq 1$. The class $\mathcal{S}^*[A,B]$ of Janowski starlike functions \cite{Jano} consists of $f\in\mathcal{A}$ satisfying the subordination $zf'(z)/f(z)\prec   (1+Az)/(1+Bz)$.  For judicious choices of $A$ and $B$, $\mathcal{S}^*[A,B]$ reduces to several widely studied subclasses of $\mathcal{A}$. For instance, the choice $\beta=(1-A)/2$ yields $\mathcal{S}^*[A,-1]=\mathcal{S}^*(\beta)$, the class which was studied in the previous section.

When $B\neq -1$, the image of  $zf'(z)/f(z)$ lies in a disc. Further, if $A$ and $B$ are close to 0, then this disc is small, and whence the class $\mathcal{S}^*[A,B]$  must be contained in the class $\mathcal{BS}(\alpha)$. This is the inclusion result given below.

\begin{theorem}\label{th7} Let $-1< B< A\leq 1$. The inclusion $\mathcal{S^{*}}[A,B]\subset \mathcal{BS}(\alpha)$ holds if  either
\begin{enumerate}
\item [(i)]
$(1-\alpha)(1+6\alpha+\alpha^2)|B|(A-B)\leq 4\alpha (1-B^2)$ and
$(1+\alpha)^2(4\alpha (A-B)^2 +B^2)\leq 4\alpha$,  or
\item [(ii)]
$(1-\alpha)(1+6\alpha+\alpha^2)|B|(A-B)\geq 4\alpha (1-B^2)$ and
  $(1-\alpha)(A-B)+|B|\leq 1$.
\end{enumerate}

\end{theorem}

\begin{proof}
Every function $f\in \mathcal{S^{*}}[A,B]$ satisfies (see \cite{ravi})
\begin{equation}\label{eqn12}
\left|\dfrac{zf'(z)}{f(z)}-\dfrac{1-A B r^2}{1-B^2 r^2}\right| \leq \dfrac{(A-B)r}{1-B^2 r^2}, \quad |z|\leq r<1.
\end{equation}
This shows that $zf'(z)/f(z)\in 	\mathbb{D}(a_f ;c_f )$  where \[ a_f  = \dfrac{1-A B}{1-B^2} \quad \text{and}\quad c_f =\dfrac{A-B}{1-B^2}.\]

We first prove the result for $B<0$. Here note that $a_f > 1$.

Case (i). Assume that  $(1-\alpha)(1+6\alpha+\alpha^2)B(A-B)\leq 4\alpha (1-B^2)$ and
$(1+\alpha)^2(4\alpha (A-B)^2 +B^2)\leq 4\alpha$. The first inequality reduces to
\[a_f \leq 1+\frac{4\alpha }{(1-\alpha)(1+6\alpha+\alpha^2)},\]
and so the result will follow if
 \[c_f^2\leq \dfrac{\sqrt{\alpha [\alpha- (1-a_f)^2 (1-\alpha ^2)^2]}
		+ \alpha (1+2 (1+\alpha)^2(1-a_f)^2)}{2\alpha (1+\alpha)^2}.\]
The latter inequality is the statement of the second inequality  $(1+\alpha)^2(4\alpha (A-B)^2 +B^2)\leq 4\alpha$.

Case (ii). Assume that  $(1-\alpha)(1+6\alpha+\alpha^2)B(A-B)\geq 4\alpha (1-B^2)$ and
  $(1-\alpha)(A-B)-B\leq 1$.  Then
\[a_f \geq 1+\frac{4\alpha }{(1-\alpha)(1+6\alpha+\alpha^2)},\]
whence the result will follow if
 \[c_f \leq 1-a_f+\dfrac{1}{1-\alpha},\]
or equivalently, when  $(1-\alpha)(A-B)-B\leq 1$.

When $B\geq 0$, the center $a_f\leq 1$,  and the proof proceeds similarly as before, and is thus omitted.
\end{proof}

We next turn our attention when the conditions in Theorem~\ref{th7} fail to hold. In this case, we seek the $\mathcal{BS}(\alpha)$-radius for the class $\mathcal{S^{*}}[A,B]$. The following result is also an extension of Theorem~\ref{th1}.

\begin{theorem}\label{th7a}
Let $0<\alpha <1$, $-1<B\leq 0$ and $B<A\leq 1$. If  neither condition (i) nor (ii) of Theorem~\ref{th7} holds, then the Booth lemniscate starlikeness radius of order $\alpha$ for the class $\mathcal{S^{*}}[A,B]$ is
	\begin{align*}
		R_{\mathcal{BS}(\alpha)}(\mathcal{S^{*}}[A,B])
		&=\begin{dcases}		
			\min \left\{1,\dfrac{2\sqrt{\alpha}}{(1+\alpha)\sqrt{4\alpha (A-B)^2 +B^2}}\right\},
			&    4A \alpha \geq  (5\alpha-1)B, \\
			\min \left\{1,\dfrac{1}{(1-\alpha)(A-B)-B}\right\}, & 4A \alpha\leq  (5\alpha-1)B.
		\end{dcases}
	\end{align*}
\end{theorem}

\begin{proof}
	The inequality \eqref{eqn12} gives $zf'(z)/f(z)\in 	\mathbb{D}(a_f(r);c_f(r))$, where
	\begin{equation*}
		a_f(r):=\dfrac{1-A B r^2}{1-B^2 r^2} \quad \text{ and }\quad c_f(r):= \dfrac{(A-B)r}{1-B^2 r^2}.
	\end{equation*}
The result follows easily for $B=0$, and so assume that $B<0$. Since
	\[a_f'(r)=-\frac{2B(B-A)r}{(1-B^2r^2)^2},\]
and $-1<B<0,$ it follows that $a_f$ is increasing with $a_f(r)\geq 1$ for $0\leq r<1$. Only a brief outline of the proof will be given here because the proof is similar to Theorem~\ref{th1}.

The numbers
	\begin{align*}
		\tilde{\rho_0}&= \dfrac{2\sqrt{\alpha}}{(1+\alpha)\sqrt{4\alpha (A-B)^2 +B^2}}\quad
		\text{and}\quad
		\rho_0 = \dfrac{1}{(1-\alpha)(A-B)-B}
	\end{align*}
	satisfy respectively the   equations
\[c_f(r)^2  = s(\alpha,a_f(r))\]
with $s(\alpha,a_f(r))$ given by \eqref{saf}, and
	\[	c_f(r)  =1-a_f(r)+\dfrac{1}{1-\alpha}.\]

Also, the number \[\rho_1=\dfrac{\sqrt{4\alpha}}{\sqrt{4\alpha B^2+(1-\alpha)(1+6\alpha+6\alpha^2)(B^2-AB)}}\]
	is the solution to the equation \[a_f(r)=1+\dfrac{4\alpha}{(1-\alpha)(1+6\alpha+\alpha^2)}.\]
	Here, the condition $\rho_1\leq \rho_0$	holds if and only if
	\begin{align}\label{eq12a}
		A &\leq \left(1-\dfrac{1-\alpha}{4\alpha}\right)B.
	\end{align}
	Thus $a_f(\rho_1)\leq a_f(\rho_0)$ if and only if \eqref{eq12a} holds.
	The result follows by an application of Lemma~\ref{ralemma} and is sharp for the function $f\in \mathcal{S}^*[A,B]$ given by $f(z)=z/(1+B z)^{(B-A)/B}$ for $B\neq 0$, while $f(z)=ze^{Az}$ for $B=0$.
\end{proof}

The result in the case $B>0$ is similar, which we state without proof.

\begin{theorem} Let $0<\alpha <1$, and $0<B<A\leq 1$. If  neither condition (i) nor (ii) of Theorem~\ref{th7} holds, then the Booth lemniscate starlikeness radius of order $\alpha$ for the class $\mathcal{S^{*}}[A,B]$ is
	\begin{align*}
		R_{\mathcal{BS}(\alpha)}(\mathcal{S^{*}}[A,B])
		&=\begin{dcases}		
			\min \left\{1,\dfrac{2\sqrt{\alpha}}{(1+\alpha)\sqrt{4\alpha (A-B)^2 +B^2}}\right\},
			&  4A \alpha \geq  (3\alpha+1)B, \\
			\min \left\{1,\dfrac{1}{(1-\alpha)(A-B)+B}\right\}, &  4A \alpha\leq  (3\alpha+1)B  .
		\end{dcases}
	\end{align*}
\end{theorem}

For $0<\beta\leq 1$, the class $S^*[\beta,-\beta]=:\mathcal{S}^*_\beta$  consists of functions $f\in \mathcal{A}$ satisfying the inequality
\[ \left|\frac{zf'(z)}{f(z)}-1\right| < \beta \left|\frac{zf'(z)}{f(z)}+1\right|. \]
Parvatham \cite{par} introduced this class in her studies on the Bernardi integral operator. The function $f(z)=z/(1-\beta z)^2$ belongs to the class $\mathcal{S}^*_\beta$. The $\mathcal{BS}(\alpha)$-radius for this class follows readily from Theorem~\ref{th7a}.

\begin{corollary}\label{cor9}
		For $0\leq \beta<1$,  the Booth lemniscate starlikeness radius of order $\alpha$  for the class $\mathcal{S}^*_\beta$  is
		\begin{align*}
			R_{\mathcal{BS}(\alpha)}(\mathcal{S}^*_\beta)
			&=\begin{dcases}	
				\min\left\{1, \dfrac{1}{\beta(3-2\alpha)  }\right\}, & \ 0\leq \alpha\leq \frac{1}{9},\\			
\min\left\{1,\dfrac{2\sqrt{\alpha}}{\beta (1+\alpha)\sqrt{ 1+16\alpha }} \right\},
				&   \frac{1}{9}\leq\alpha \leq 1.
			\end{dcases}
		\end{align*}
\end{corollary}

It is worthy to note that for $\beta=1$, Corollary~\ref{cor9} reduces to the one given by \eqref{eqn31b}.

For $0\leq \beta <1$, the class $\mathcal{S}^*[1-\beta,0]:=\mathcal{S}^*[\beta]$ consists of functions $f\in\mathcal{A}$ satisfying the inequality
\[\left|\frac{zf'(z)}{f(z)}-1\right| <1-\beta.\]
Clearly, $\mathcal{S}^*[\beta]\subset \mathcal{S}^*(\beta)$ and the function $f(z)=ze^{(1-\beta)z}$ belongs to the class $\mathcal{S}^*[\beta]$.  This class was introduced and studied by Fournier \cite{four}, and we state its Booth lemniscate starlikeness radius.
\begin{corollary}
		For $0\leq \beta<1$,  the Booth lemniscate starlikeness radius of order $\alpha$  for the class $\mathcal{S}^*[\beta]$  is
		\begin{align*}
			R_{\mathcal{BS}(\alpha)}(\mathcal{S}^*[\beta]) 	& =\min\left\{1; 		
\dfrac{1}{ (1+\alpha)(1-\beta)}\right\}.
		\end{align*}
In particular, $\mathcal{S}^*[\alpha/(1+\alpha)]\subset \mathcal{BS}(\alpha) $.
\end{corollary}

\end{document}